\documentclass{amsart}

\usepackage{amsmath,amsthm}
\usepackage{amssymb}
\usepackage{url}
\usepackage{xypic}
\usepackage{enumerate}

\newtheorem{theorem}{Theorem}[section]
\newtheorem{lemma}[theorem]{Lemma}

\newtheorem{algorithm}[theorem]{Algorithm}

\newtheorem{corollary}[theorem]{Corollary}

\theoremstyle{definition}
\newtheorem{definition}[theorem]{Definition}

\theoremstyle{remark}

\numberwithin{equation}{section}

\usepackage{multirow}

\begin{document}

\title{Realization of Modular Galois Representations in Jacobians of modular curves}

\author{Peng Tian}
\address{Department of Mathematics, East China University of Science and Technology, Shanghai, China 200231}
\email{tianpeng@ecust.edu.cn}
\thanks{The author is supported by NSFC Grant \#11601153.}






\begin{abstract}
In this paper, we propose an improved algorithm for computing mod $\ell$ Galois representations associated to eigenforms of arbitrary levels prime to $\ell$. Precisely, we present a method to find the Jacobians of modular curves which have the smallest possible dimensions in a well-defined sense to realize the modular Galois representations. This algorithm also works without the assumption $\ell \ge k-1$, where $k$ are the weights of the eigenforms.
\end{abstract}

\maketitle



\section{Introduction} \label{sec:introduction}

In the book \cite{book}, S. Edixhoven and J.-M. Couveignes propose a polynomial time algorithm to compute the mod $\ell$ Galois representations $\rho_{f,\ell}$ associated to level one eigenforms. P. Bruin \cite{bruin} generalizes the algorithm  and applies on eigenforms of arbitrary levels.

Let  $f\in S_{k}(\varGamma_{1}(N), \varepsilon) $ be an eigenform and $\ell$ be a prime with $\ell \ge k-1$.  Let $N'=N\ell$ if $k>2$ and $N'=N$ if $k=2$.

Let $J_{1}$ denote the Jacobian of the modular curve $X_{1}(N')$  associated to $\varGamma_{1}(N')$. Let $\mathbb{T}\subseteq$ End$J_{1}$ be the Hecke algebra associated to $S_{2}(\varGamma_{1}（(N'))$ and $\mathfrak{m}$ be the maximal ideal associated to $f$. Then it is well known that the $(\mathbb{T}/\mathfrak{m})[Gal(\overline{\mathbb{Q}}|\mathbb{Q})]$-module $J_{1}(\overline{\mathbb{Q}})[\mathfrak{m}]$ is a non-zero finite direct sum of copies of  $\rho_{f,\ell}$. The computations of $\rho_{f,\ell}$ boil down to producing 
the representation
$$
\rho_{J_{1}(\overline{\mathbb{Q}})[\mathfrak{m}]}: Gal (\mathbb{\overline{Q}}/\mathbb{Q}) \rightarrow \mathrm{Aut}_{\mathbb{T}/\mathfrak{m}}(J_{1}(\overline{\mathbb{Q}})[\mathfrak{m}]).
$$

S. Edixhoven and J.-M. Couveignes \cite{book} propose a method to efficiently compute $\rho_{J_{1}(\overline{\mathbb{Q}})[\mathfrak{m}]}$. They prove that  $\rho_{f}$ can be described by a certain polynomial $P_{f}\in \mathbb{Q}[x]$ whose splitting field is the fixed field $L$ of ker($\rho_{f}$). The polynomial can be computed by approximately evaluating the points of $J_{1}(\overline{\mathbb{Q}})[\mathfrak{m}]$.

However, in practice, the most time-consuming part of the algorithm is to evaluate $J_{1}(\overline{\mathbb{Q}})[\mathfrak{m}]$ and it heavily depends on the dimension of $J_1$. In the paper \cite{tian}, the author presents an improvement of this algorithm in the cases that $\ell \ge k-1$ and  $f$ has level one. In these cases,  one can do the computations  with the Jacobian $J_{\varGamma_{H}}$ of $X_{\varGamma_{H}}$ rather than  $J_1$, where $X_{\varGamma_{H}}$ is a modular curve of smaller genus with $\varGamma_{1}(\ell)\lneq \varGamma_{H} \leq \varGamma_{0}(\ell)$. The explicit computations of evaluating $J_{1}(\overline{\mathbb{Q}})[\mathfrak{m}]$ can be greatly reduced by this improved algorithm.

In this paper, we generalize the improved algorithm of \cite{tian} to the cases that $\ell\ge5$ may be any prime without the assumption $\ell \ge k-1$ and the eigenform $f$ has arbitrary level prime to $\ell$.

We firstly propose an algorithm, for a  normalized eigenform $f\in S_{k}(\varGamma_{1}(N))$, to find an integer $i$, a congruence subgroup $\varGamma_{H}$ and a normalized eigenform $f_2\in S_2(\varGamma_{H})$, such that $\rho_{f,\ell}$ is isomorphic to $\rho_{f_2,\ell}\otimes \chi^{i}_{\ell}$. We also show that the subgroup $\varGamma_{H}$ produced by this algorithm is the largest possible congruence subgroup with $\varGamma_{1}(N')\subseteq\varGamma_{H}\subseteq\varGamma_{0}(N')$, on which such eigenform $f_2$ exists.

Let $J_{\varGamma_{H}}$ be the Jacobian of the modular curve $X_{\varGamma_{H}}$  associated to $\varGamma_{H}$. We then demonstrate that $J_{1}(\overline{\mathbb{Q}})[\mathfrak{m}]$ is a $2$-dimensional subspace of $J_{\varGamma_{H}}[\ell]$ and the representation $\rho_{J_{1}(\overline{\mathbb{Q}})[\mathfrak{m}]}$ is a subrepresentation of $J_{\varGamma_{H}}[\ell]$. 
This allows us to evaluate the points of $J_{1}(\overline{\mathbb{Q}})[\mathfrak{m}]$  by working with the Jacobian $J_{\varGamma_{H}}$, which has the smallest possible dimension in the sense that $\varGamma_{H}$ is the largest possible congruence subgroup.

As examples, we do explicit computations to calculate the eigenforms $f_2$ and list the dimensions of  $J_{1}(N')$ and  $J_{\varGamma_{H}}$  in the cases with $\ell$ up to $13$ and $N$ up to $6$.

In the end, we discuss the case that $k>2$ and $f\in S_{k}(\varGamma_{0}(N))$ is an eigenform on $\varGamma_{0}(N)$. To be precise, we prove  that the index $[\varGamma_{H}:\varGamma_1(N\ell)]$ of $\varGamma_{1}(N\ell)$  in $\varGamma_H$ is equal to $\frac{\phi(N\ell)\cdot \mathrm{gcd}(\ell-1,k-2-2i)}{\ell-1}$. Then, we apply this result to give the criteria for the occurrence of $\varGamma_H=\varGamma_{0}(N\ell)$. As a consequence of the criteria, it can be shown that, for a normalized eigenform $f\in S_{\ell+1}(\varGamma_{1}(N))$, the existence of  a normalized eigenform $f_2\in S_2(\varGamma_{0}(N\ell))$  with $\rho_{f,\ell}\cong\rho_{f_2,\ell}$ is equivalent to  $f\in S_{\ell+1}(\varGamma_{0}(N))$. 

The rest of this paper is organized as follows. In Section \ref{sec:preliminary}, we recall the computations of modular Galois representations. In Section \ref{sec:character}, we define Teichm\"uller lifting of Dirichlet characters and give the results that play important roles in the next section. Our main results and algorithms are presented in Section \ref{sec:algorithm}. In Section \ref{sec:gamma0}, we apply our main results to the case of eigenforms on $\varGamma_{0}(N)$.

\section{Computations of modular Galois representations} \label{sec:preliminary}
We let $\ell$ denote a prime with $\ell\ge5$ and $v$ be a place dividing $\ell$ of the field of algebraic numbers $\overline{\mathbb{Q}}$. The residue field of $v$ is denoted by $\overline{\mathbb{F}}_\ell$ and it is then the algebraic closure of the prime field $\mathbb{F}_\ell$.

For any positive integer $n$, the congruence subgroups $\varGamma_{0}(n)$ and $\varGamma_{1}(n)$  respectively are
$$\varGamma_{0}(n)=
\left\lbrace 
\begin{pmatrix}
a & b\\
c & d
\end{pmatrix}  \in SL(2,\mathbb{Z}) \ \mid \   c\equiv 0 \ (\mathrm{mod} \ n) 
\right\rbrace \ \ \ \mathrm{and}$$
$$\varGamma_{1}(n)=
\left\lbrace 
\begin{pmatrix}
a & b\\
c & d
\end{pmatrix}  \in SL(2,\mathbb{Z}) \ \mid \   c\equiv 0 \ (\mathrm{mod} \ n) \ \ \mathrm{and} \  \ d\equiv 1 \ (\mathrm{mod} \ n) 
\right\rbrace.$$

Now let $N>0$ and $k\ge2$ be  integers. Let $q=q(z)=e^{2\pi iz}$ and  $f(z) = \sum_{n>0} a_{n} (f) q^{n} \in S_{k}(\varGamma_{1}(N), \varepsilon) $ be a normalized eigenform of weight $k$ and level $N$, with nebentypus character $\varepsilon $. Let $K_f$ be the number field of $f$, which is obtained by adjoining all the Fourier coefficients $ a_{n}(f) $ of the $q$-expansion of $f$ to $\mathbb Q$. Let $\lambda$ be a prime of $K_{f}$ lying over $\ell$.  Then  P. Deligne \cite{deligne71} proves  the following well known theorem:

\begin{theorem}  \label{thm:deligne}
	There exists a unique (up to isomorphism) continuous
	semi-simple  representation
	\begin{center}
		$\rho_{f,\lambda}: Gal(\overline{\mathbb{Q}}|\mathbb{Q}) \rightarrow GL_{2}(\overline{\mathbb{F}}_{\ell})$.
	\end{center}
	which is unramified outside $N\ell$ and such that for all primes $p\nmid N\ell$ the characteristic polynomial of  $\rho_{f,\lambda} (Frob_{p}) $ satisfies
	\begin{equation} \label{charpol}
	charpol( \rho_{f,\lambda} (Frob_{p}))\equiv x^2-a_{p}(f)x+ \varepsilon(p)p^{k-1} \mod{\lambda}.
	\end{equation}
\end{theorem} 

We also let $\rho_{f,\ell}$ denote the representation $\rho_{f,\lambda}$ when the prime $\lambda$ is not involved in our discussion.

When we say computing a Galois representation $\rho_{f,\ell}$, it means to give:
\begin{enumerate}
	\item the finite Galois extension $L$ over $\mathbb{Q}$, such that  $\rho_{f,\ell}$ factors as
	\begin{displaymath} 
	\xymatrix{Gal (\mathbb{\overline{Q}}|\mathbb{Q}) \ar[rr]^{\rho_{f,\ell}}\ar[dr]_{\pi}  &   & GL_{2}(\overline{\mathbb{F}}_{\ell})  \\
		&  Gal ( L|\mathbb{Q}) \ar[ur]^{\varPsi} &
	}
	\end{displaymath}
	where $\pi$ is the natural surjection and $\varPsi$ is an injection.
	\item   the matrix of $\rho_{f,\ell}(\sigma)$ for each $\sigma\in  Gal ( L|\mathbb{Q})$.
\end{enumerate}

In the book \cite{book}, S. Edixhoven and J.-M. Couveignes propose a polynomial time algorithm to compute  $\rho_{f,\ell}$ for level one eigenforms. In his Ph.D thesis \cite{bruin}, P. Bruin generalizes the algorithm  and applies on eigenforms of arbitrary levels. 

Let  $f\in S_{k}(\varGamma_{1}(N), \varepsilon) $ be a normalized eigenform. Let $\ell$ be a prime number with $\ell \ge k-1$ and $v$ be a place dividing $\ell$ of the field of algebraic numbers $\overline{\mathbb{Q}}$. Let $N'=N\ell$ if $k>2$ and $N'=N$ if $k=2$.
We let $J_{1}$ be the Jacobian of the modular curve $X_{1}(N')$  associated to $\varGamma_{1}(N')$. Let $\mathbb{T}\subseteq$ End$J_{1}$ be the Hecke algebra generated by the diamond and Hecke operators over $\mathbb Z$  and let $\mathcal{I}_{f}$ be the  ring homomorphism $\mathcal{I}_{f}: \mathbb{T} \rightarrow  \overline{\mathbb{F}}_\ell,$ given by $\langle d \rangle \mapsto \varepsilon(d)$ and $T_{n} \mapsto a_{n}(f) \mod v$.
Let $\mathfrak m_{f}$ denote the kernel of $\mathcal{I}_{f}$ and if we put 
\begin{displaymath} 
V_{f}=J_{1}(\overline{\mathbb{Q}})[\mathfrak{m}_{f}]=\{x \in J_{1}(\overline{\mathbb{Q}}) \ | \ tx=0 \ \mathrm{for}  \ \mathrm{all} \ t \ \mathrm{in} \ \mathfrak m_{f} \},
\end{displaymath} 
then $V_{f}$ is a $\mathbb{T}/\mathfrak{m}$-vector subspace of the $\ell$-torsion points $J_{1}(\overline{\mathbb{Q}})[\ell]$ of $J_{1}(\overline{\mathbb{Q}})$, and moreover, $V_{f}$ is a non-zero finite direct sum of copies of  $\rho_{f,\ell}$. The number of the copies of  $\rho_{f,\ell}$ is called \textit{multiplicity} of $\rho_{f,\ell}$ and  the representation  $\rho_{f,\ell}$ is called a \textit{multiplicity one representation} if  its multiplicity is equal to one. 

Namely, for a multiplicity one representation  $\rho_{f,\ell}$, the vector space $V_{f}$ has dimension 2 and $\rho_{f,\ell}$ is isomorphic to
the representation
$$
\rho_{V_f}: Gal (\mathbb{\overline{Q}}/\mathbb{Q}) \rightarrow \mathrm{Aut}_{\mathbb{T}/\mathfrak{m}_{f}}(V_f).
$$
If the multiplicity is larger than one, then $\rho_{f,\ell}$ is isomorphic to any simple constituent of $\rho_{V_f}$.

 Following the work of Mazur, Ribet, Gross, Edixhoven and Buzzard, we know that  in most cases the multiplicity is equal to one. More precisely, an irreducible modular Galois representation $\rho_{f,\ell}$ is a multiplicity one  representation, except the case that all the following hypothesis are simultaneously satisfied:
\begin{enumerate}
	\item $k=\ell$;
	\item $\rho_{f,\ell}$ is unramified at $\ell$;
	\item $\rho_{f,\ell} (Frob_{\ell})$ is a scalar matrix.
\end{enumerate}
For details, we refer to \cite[Section 3.2 and 3.3]{ribetstein}.

 Thus, to compute modular Galois representation $\rho_{f,\ell}$, it suffices to compute the representation $\rho_{V_f}$ (in the very few cases that the multiplicities are larger than one, it is in fact to compute any simple constituent of $\rho_{V_f}$).
 
 The method provided by Edixhoven and Couveignes to compute $\rho_{V_f}$ is to evaluate a suitable  polynomial $P_{V_f}\in \mathbb Q[X]$ whose splitting field is the fixed field of $\rho_{V_f}$. More precisely,  we can take the polynomial to be
\begin{equation}\label{polynomial}
P_{V_f}(x) = \prod_{Q\in V_f-\{0\}}(x-\sum_{i=1}^g h(Q_i)),
\end{equation}
for some suitable function $h(x)$ in the function field of $X_1(N')$, where  $g$ is the genus of $X_1(N')$ and $Q_i$ are the points on $X_1(N')$ such that $Q=\sum_{i=1}^g (Q_i) -g\cdot(O)$ as divisors on $X_1(N')$ via the Abel-Jacobi map.

In \cite{book}, the authors propose two methods to efficiently evaluate the points $Q\in V_f-\{0\}$, either over complex numbers or modulo sufficiently many small prime numbers to reconstruct $V_f$.
 In each of the methods, however, it requires high precisions to approximate the points of $V_f$. 
Consequently, it always takes quite much time to obtain the polynomial $P_{V_f}$ in practice. It is known that the required precisions and calculations increase very rapidly with the growth of the dimension of the Jacobian.
Therefore, the calculations can be largely decreased if we can work with a Jacobian which has a smaller dimension.

Modular symbol is a very effective tool in the computations of modular forms. We refer to \cite{Stein07} for the theory of computing modular forms via modular symbols.

\section{Dirichlet characters and Teichm\"uller lifting}  \label{sec:character}

In this section, we first recall Dirichlet characters and then define the Teichm\"uller lifting which has an important role in proving our main results.
 
\subsection{Dirichlet characters}

Let $n$ be a positive integer. A Dirichlet character modulo $n$ is a homomorphism of multiplicative groups:
\begin{displaymath}
	\varepsilon: \ (\mathbb{Z}/n\mathbb{Z})^{\ast}\rightarrow \mathbb{C}^{\ast}.
\end{displaymath} 

For two Dirichlet characters $\varepsilon_{1}$ and $\varepsilon_{2}$ modulo $n$, the product character  $\varepsilon_{1}\varepsilon_{2}$, which is also a Dirichlet character modulo $n$, is defined by 
$$(\varepsilon_{1}\varepsilon_{2})(x)=\varepsilon_{1}(x)\varepsilon_{2}(x).$$ 

Let $d$ be a positive divisor of $n$. Let $\pi_{n,d}$ be the canonical projection
\begin{displaymath} 
	\pi_{n,d}: \ \ (\mathbb{Z}/n\mathbb{Z})^{\ast}\rightarrow (\mathbb{Z}/d\mathbb{Z})^{\ast} , \ \ \ \ \ \ \
	x \mod n  \ \rightarrow \ \ \ \  x\mod d.
\end{displaymath}

Then we have
\begin{lemma} \label{lemmaofnd}
	Let $d,n$ be two integers with $d|n$. Then the canonical homomorphism  $\pi_{n,d}$  is surjective.
\end{lemma}
\begin{proof}
	For any integer $x$ with gcd$(x,d)=1$, take $x'=x+d\prod\limits_{\substack{primes \ p|n  \\ \ \ \ \ \ \ \ \  p\nmid x}}p$. Then we have gcd$(x',n)=1$ and $x'\equiv x\mod d$.
\end{proof}

Each Dirichlet character $\varepsilon$ modulo $d$ can lift to a unique Dirichlet character $\varepsilon_{ind}$ modulo $n$ such that
$$\varepsilon_{ind}=\varepsilon \circ \pi_{n,d},$$
and the character $\varepsilon_{ind}$ is said to be \textit{induced} by $\varepsilon$. Equivalently, the character $\varepsilon_{ind}$ is trivial on the normal subgroup 
$$\{x\in (\mathbb{Z}/n\mathbb{Z})^{\ast}: x\equiv 1 \mod d\}.$$
Thus we know 
\begin{equation} \label{inducedkernal}
\mathrm{ker}(\varepsilon)=\pi_{n,d}(\mathrm{ker}(\varepsilon_{ind})):=\{x \ (\mathrm{mod} \ d) \ | \  
\mathrm{gcd}(x,n)=1 \ \ \mathrm{and} \ \ \varepsilon_{ind}(x)=1 \}.
\end{equation}

The \textit{conductor} of a Dirichlet character $\varepsilon$ is defined to be the smallest positive divisor $d$ of $n$ such that $\varepsilon$  is induced by some character  modulo $d$. A Dirichlet character is trivial if and only if its conductor is $1$. 

Moreover, it can be shown that the conductor of a Dirichlet character $\varepsilon$ is the greatest common divisor of all divisors $d$ of $n$ such that $\varepsilon$ is induced by a character modulo $d$.

\subsection{Teichm\"uller lifting}

	For a positive integer $n$, we let $\zeta_{n}$ denote the primitive $n$-th root of unity. To give the definition of Teichm\"uller lifting of Dirichlet character, we need

\begin{lemma} \label{lem:distinctrootsofunity}
	Let $n$ be a positive integer. Let $\ell$ be a prime number with $\ell \nmid n$ and $v$ be a place dividing $\ell$ of $\overline{\mathbb{Q}}$. Then the n-th roots of unity are distinct modulo $v$.
\end{lemma}
\begin{proof}
    It is known that
	$$ \prod_{1\le k \le n-1}(x-\zeta_{n}^{k})=\frac{x^{n}-1}{x-1}.$$
	Now set $x=1$ and we obtain 
	\begin{equation} \label{equ:n}
      n = \prod_{1\le k \le n-1}(1-\zeta_{n}^{k}).
    \end{equation}
	Suppose that there exist $0\le i < j \le n-1$ such that $\zeta_{n}^{i} \equiv\zeta_{n}^{j} \mod v$. It follows that $v| (1-\zeta_{n}^{j-i})$ and hence $v|n$ by (\ref{equ:n}). This leads a contradiction since gcd$(\ell, n)=1$  and $v$ is a place dividing $\ell$.
	Therefore the $n$-th roots of unity are distinct modulo $v$.
\end{proof}

Now let $r$ be a positive integer and $\ell$ be a prime number. Take $q=\ell^r$. Let $L=\mathbb{Q}(\zeta_{\ell}, \zeta_{q-1})$ and $O_{L}$ be the integer ring of $L$. Let $\mathfrak{l}=v\cap O_L$ be a prime of $L$ lying over $\ell$. Since gcd$(\ell,q-1)=1$ and $r$ is the smallest positive integer such that 
$$\ell^{r}\equiv 1 \mod q-1,$$
by the theory of cyclotomic field, we know that  $L=\mathbb{Q}(\zeta_{\ell(q-1)})$ and the inertia degree of $\mathfrak{l}$ over $\ell$ is $r$, i.e., the residue field $\mathbb{F}_{\mathfrak{l}}=O_{L}/\mathfrak{l}$ has order $q$. Let  $\mu_{q-1}=\{\zeta_{q-1}^{j}|0\le j\le q-2\}$ be the group of $(q-1)$-st roots of unity. Then we have  

\begin{lemma} \label{thm:liftingofcharacter}
	The reduction modulo $\mathfrak{l}$ restricted on  $\mu_{q-1}$ 
	\begin{displaymath}
		\bar \ : \  \mu_{q-1} \longrightarrow \mathbb{F}_{\mathfrak{l}}^*, \ \ \ \ \ \ \ \ \ \ \ \ \ a  \ \rightarrow \  \bar a = a \bmod \mathfrak{l}, 
	\end{displaymath} 
	is a group isomorphism.
\end{lemma}
\begin{proof}
	The reduction modulo $\mathfrak{l}$ is obviously a group homomorphism. Since  gcd$(q-1,\ell)=1$, it follows from Lemma \ref{lem:distinctrootsofunity} that the elements of  $\mu_{q-1}$ are distinct modulo $\mathfrak{l}$, and this implies that the homomorphism is injective. Note that both $\mu_{q-1}$ and $\mathbb{F}_{\mathfrak{l}}^*$  have $q-1$ elements, and it shows that the homomorphism is in fact an isomorphism.
\end{proof}

Let $\ell$ be a prime number  and $v$ be a place dividing $\ell$ of $\overline{\mathbb{Q}}$. Let $n$ be an positive integer and $\varepsilon$ be a Dirichlet character modulo $n$.
Let $E$ denote the number field which is obtained by adjoining all the values of $\varepsilon$ to $\mathbb{Q}$.  For  a prime $\lambda$ of $E$ lying over $\ell$, let  $\bar\varepsilon$ denote  the reduction of $\varepsilon$ mod $\lambda$. Then we have

\begin{theorem} \label{prop:teichmullerlifting}
	There exists a Dirichlet character $\mathrm{T}(\bar\varepsilon)$ modulo $n$ which satisfies:
	\begin{enumerate}
		\item $\mathrm{T}(\bar\varepsilon) \equiv \varepsilon \mod v$; \ and   
		\item  $\mathrm{ker}(\mathrm{T}(\bar\varepsilon))=\mathrm{ker}(\bar\varepsilon)$.
	\end{enumerate}
\end{theorem}

\begin{proof}
	Let $O_E$ be the integer ring of $E$ and $\mathbb{F}_{\lambda}=O_{E}/\lambda$ be the residue field. Then $\bar{\varepsilon}$ factors as
	\begin{displaymath}
		\bar\varepsilon : \  (\mathbb{Z}/n\mathbb{Z})^{\ast} \xrightarrow{\ \ \ \varepsilon \ \ \ } O_E^* \xrightarrow{\mod \lambda \ } \mathbb{F}_{\lambda}^*.
	\end{displaymath} 
	
	We let $q=\#\mathbb{F}_{\lambda}$, and then by Lemma \ref{thm:liftingofcharacter}, we have a group isomorphism
	\begin{displaymath}
		\omega: \mathbb{F}_{\lambda}^*\longrightarrow  \mu_{q-1},
	\end{displaymath} 
	which  is the inverse of the isomorphism in Lemma \ref{thm:liftingofcharacter}. Thus it satisfies $\omega(x)\equiv x\mod v$. Composing $\omega$ with $\bar\varepsilon$, we obtain a Dirichlet character T$(\bar\varepsilon)=\omega \circ \bar{\varepsilon}$ modulo $n$ which satisfies:
	\begin{enumerate}
		\item $\omega \circ \bar{\varepsilon} \equiv \varepsilon \mod v$; \ and   
		\item  ker$(\omega \circ \bar{\varepsilon})=$ker$(\bar\varepsilon)$.
	\end{enumerate}
\end{proof}

\begin{definition}  \label{def:teichmullerlifting}
	A Dirichlet character $\mathrm{T}(\bar\varepsilon)$ which satisfies the conditions $(1)$ and $(2)$ in Theorem \ref{prop:teichmullerlifting} is called a Teichm\"uller lifting of $\bar{\varepsilon}$. 
\end{definition}

Now let $\phi(n)=n \prod\limits_{p|n} (1-\frac{1}{p})$ be the Euler's totient function. Then we have
\begin{lemma}  \label{lem:teichmullerliftingforlandphicoprime}
	If gcd$(\ell,\phi(n))=1$, then $\varepsilon$ is a Teichm\"uller lifting of $\bar{\varepsilon}$.
\end{lemma}
\begin{proof}
   It suffices to prove  ker$(\bar\varepsilon)\subseteq$ ker$(\varepsilon)$. For any $x\in\mathrm{ker}(\bar\varepsilon)$, we have $\overline{\varepsilon(x)}\equiv 1 \mod v$. Since  gcd$(\ell,\phi(n))=1$ and $\varepsilon(x)$ is a $\phi(n)$-th root of unity, it follows from Lemma \ref{lem:distinctrootsofunity} that $\varepsilon(x)=1$. Hence we have $x\in$  ker$(\varepsilon)$ and ker$(\bar\varepsilon)\subseteq\mathrm{ker}(\varepsilon)$.
\end{proof}

\section{Realization of modular Galois representations in Jacobians of the smallest possible dimensions}  \label{sec:algorithm}
In this section, we describe our method to find the Jacobians of modular curves, which have the smallest possible dimensions in a well-defined sense, to realize the modular Galois representations. As examples, we will give the explicit results of the cases with $\ell$ up to $13$ and $N$ up to $6$.

In this section, we follow the notation of Section \ref{sec:preliminary}. Moreover,  the following notations may be used in the rest of this paper.

Let $n$ be a positive integer and $H$ be a subgroup of $(\mathbb{Z}/n\mathbb{Z})^{\ast}$. Then we let $\varGamma_{H}(n)$ denote the congruence subgroup
$$\varGamma_{H}(n)=
\left\lbrace 
\begin{pmatrix}
a & b\\
c & d
\end{pmatrix}  \in SL(2,\mathbb{Z}) \ \mid \   c\equiv 0 \ (\mathrm{mod} \ n) \ \ \mathrm{and} \  \ d \ (\mathrm{mod} \ n)  \in H
\right\rbrace.$$

Let $\varphi_{n}$ denote the surjection:
\begin{displaymath} 
\varphi_{n} :  \ \varGamma_{0}(n) \twoheadrightarrow (\mathbb{Z}/n\mathbb{Z})^{\ast} , \ \ \ \ \ \ \
\begin{pmatrix} a & b\\ c & d \end{pmatrix} \rightarrow \overline{d}. 
\end{displaymath}
Then we know the kernel of $\varphi_{n}$ is $\varGamma_{1}(n)$ and the preimage $\varphi_{n}^{-1}(H)$ of $H$ under $\varphi_{M}$ is $\varGamma_{H}(n)$.

\subsection{Twists of modular Galois representations}

In order to discuss the case with $\ell<k-1$, we first give some results on the twists of modular Galois representations by the cyclotomic character.

Let $\theta=q\frac{d}{dq}$  be the classical differential operator. If  $f\in S_{k}(\varGamma_{1}(N),\varepsilon)$ is an eigenform, then $\theta f \in S_{k+\ell+1}(\varGamma_{1}(N),\varepsilon)$ is also an eigenform. Suppose $f = \sum_{n>0} a_{n} (f) q^{n}$. Then we know $\theta f$ has $q$-expansion $\sum_{n>0} na_{n} (f) q^{n}$. It follows from Theorem \ref{thm:deligne} that
$$  \rho_{\theta f, \ell}=  \rho_{f,\ell} \otimes \chi_{\ell}, $$
where $\chi$ is the \textit{mod $\ell$ cyclotomic character}. 

For the case with $\ell<k-1$, the Galois representation associated to $f$ can be reduced to the case with $\ell\ge k-1$ by twisting. In fact we have the following result which is a corollary of \cite[Theorem 3.4]{ediweight}.

\begin{theorem} \label{thm:twist}  
	Let $\ell\ge 5$ be a prime number, $N>0$ an integer prime to $\ell$, and $k\ge 2$.  Let $f\in S_{k}(\varGamma_{1}(N),\varepsilon)$ be an eigenform  and $\lambda$ be a prime of $K_f$ lying over $\ell$. Suppose the representation $\rho_{f,\lambda}$ is irreducible and $a_{1}(f) \not \equiv 0 \mod \lambda$. Then there exist integers $i$ and $k'$ with $0\le i\le \ell-1, \ k'\le \ell+1$,  a newform $g\in S_{k'}(\varGamma_{1}(M))$ for some $M|N$, and a prime   $\mathfrak{l}$ of $K_g$ lying over $\ell$, such that $\rho_{f,\lambda}$ is isomorphic to $\rho_{g,\mathfrak{l}}\otimes \chi^{i}_{\ell}$.
	
	Moreover, the  character of $f$ is induced by the character of $g$.
\end{theorem}
\begin{proof} 
	By \cite[Theorem 3.4]{ediweight}, we have  $i$ and $k'$ with $0\le i\le \ell-1, \ k'\le \ell+1$, and an eigenform $g'\in S_{k'}(\varGamma_{1}(N),\varepsilon)$, and a prime   $\mathfrak{l}$ of $K_{g'}$ lying over $\ell$, such that $\rho_{f,\lambda}$ is isomorphic to $\rho_{g',\mathfrak{l}}\otimes \chi^{i}_{\ell}$.
	
    Since the representation $\rho_{f,\lambda}$ is irreducible, so is  $\rho_{g',\mathfrak{l}}\cong \rho_{f,\lambda}\otimes \chi^{-i}_{\ell}$. It follows that $g'$ is a cuspidal eigenform.  By $a_1(f) \not \equiv 0 \mod \lambda$, we know $g'$ is nonzero, and thus we have $a_1(g')\ne0$. Let $g''=(a_1(g'))^{-1}g'$ be the normalized eigenform and then we know $g''\in S_{k'}(\varGamma_{1}(N),\varepsilon)$ and $\rho_{f,\lambda}\cong\rho_{g'',\mathfrak{l}}\otimes \chi^{i}_{\ell}$. 
    
    By \cite[Theorem 1.2]{ribetnebentypus}, there is a newform $g\in S_{k'}(\varGamma_{1}(M),\varepsilon_g)$ for some divisor $M$ of $N$ such that $a_{n}(g)=a_n(g'')$ and $\varepsilon(n)=\varepsilon_g(n)$ for all $n$ with gcd$(n,N)$=1. Therefore, we have $\rho_{f,\lambda}\cong\rho_{g,\mathfrak{l}}\otimes \chi^{i}_{\ell}$ and  the  character $\varepsilon$ is induced by the character of $g$.
\end{proof}

Since $g\in S_{k'}(\varGamma_{1}(M))$ is naturally a normalized eigenform on $\varGamma_{1}(N)$, we can determine $i$, $k'$ and $g$ in Theorem \ref{thm:twist} by the following theorem.

\begin{theorem} \label{thm:kandi}  
	Let $f$ and $g$ be two normalized eigenforms on $\varGamma_{1}(N)$ of weight $k$ and $k'$, respectively.  Let $\ell$ be a prime number. Let $\lambda$  and $\mathfrak{l}$  be  primes of $K_f$ and $K_g$ lying over $\ell$.  Let $i$ be an integer with $0\le i \le \ell-1$.  Then  $\rho_{f,\lambda}$ is isomorphic to $\rho_{g,\mathfrak{l}}\otimes \chi^{i}_{\ell}$ if and only if $k\equiv k'+2i \mod \ell-1$ and  $a_{p}(f) = p^{i}a_{p}(g)$ in $\overline{\mathbb{F}}_{\ell}$ for all primes $p$ with  $p \le \frac{[SL_{2}(\mathbb{Z}):\varGamma_{1}(N)]}{12}\cdot(\ell^2-1+max\{k,k'\})$.
\end{theorem}
\begin{proof}
	See \cite[Theorem 3.5]{bruin}.
\end{proof}

Since the  Dirichlet character $\varepsilon$ of $f$ is induced by the character of $g$ as stated in Theorem \ref{thm:twist}, we know the divisor $M$ of $N$ should be divisible by the conductor of $\varepsilon$. Moreover,	if we suppose  $\ell\ge k-1$, the integers $i$ and $k'$ as given in Theorem \ref{thm:twist} can be taken to be $0$ and $k$, respectively. Then we can write down the algorithm for a normalized eigenform $f$ to find such an integer $i$ and a  newform $g$ as given in Theorem \ref{thm:twist}.

\begin{algorithm}    \label{alg:twistforsmalll}
	Let $\ell\ge 5$ be a prime number, $N>0$ an integer prime to $\ell$, and $k\ge 2$. Let $f$ be a normalized eigenform on $\varGamma_{1}(N)$ of weight $k$ and $\lambda$ be a prime of $K_f$ lying over $\ell$.  Let $d$ be the conductor of the Dirichlet character of $f$. This algorithm outputs integers $i$ and $k'$ with $0\le i\le \ell-1, \ k'\le \ell+1$,   a newform $g\in S_{k'}(\varGamma_{1}(M)$ for some divisor $M$ of $N$,  and a prime   $\mathfrak{l}$ of $K_g$ lying over $\ell$, such that $\rho_{f,\lambda}$ is isomorphic to $\rho_{g,\mathfrak{l}}\otimes \chi^{i}_{\ell}$.
\begin{enumerate}[$1.$]
	\item Compute $B=\frac{[SL_{2}(\mathbb{Z}):\varGamma_{1}(N)]}{12}\cdot(\ell^2-1+max\{k,k'\})$ and  $a_{p}(f)$ for all primes $p$ with  $p \le B$.
	\item  Compute the set $S$ consisting of all the divisors of $N$ that is divisible by $d$.
	\item If $\ell\ge k-1$, set $i\leftarrow0$, $k'\leftarrow k$ and go to step $8$. Otherwise go to step $4$.
    \item Set $i\leftarrow0$. 
	\item   Set $k' \leftarrow2$. 
	\item If $k'>\ell+1$, go to step $10$. Otherwise go to step $7$.
	\item If  $k\equiv k'+2i \mod \ell-1$, go to step $8$. Otherwise,  go to step $9$.
	\item  If $S$ is empty, go to step $9$. Otherwise, take $M$ in $S$ and do:
	      \begin{enumerate}

	\item   Compute all normalized newforms $F$ in  $S_{k'}(\varGamma_{1}(M))$ using modular symbols. 
	\item   For each $g$ in $F$, do:
	      \begin{enumerate}
	      	\item Compute  $p^{i}a_{p}(g)$  for all primes $p$ with  $p\le B$ and compute primes $P$ of the composed field $K_{f}K_{g}$ lying over $\ell$. 
	      	\item 	         If there is a prime  $\mathfrak{l} \in P$ such that  $a_{p}(f) \equiv p^{i}a_{p}(g) \mod \mathfrak{l}$ for all primes $p$ with  $p    \le B$,  put $\mathfrak{l}=\mathfrak{l}\cap K_{g}$ and then output $i$, $k'$, $M$, $g$, $\mathfrak{l}$, and terminate. 
	      \end{enumerate}
     \item Set $S\leftarrow S-\{M\}$ and go to step $8$.  
          \end{enumerate}
	\item   Set  $k' \leftarrow k'+1$ and go to step $6$. 
	\item Set $i\leftarrow i+1$ and go to step $5$.
\end{enumerate}
\end{algorithm}

\subsection{The largest possible congruence subgroup associated to $\rho_{f,\ell}$}
Let $N>0$ be an integer and $f$ be an eigenform of level $N$. In this subsection, we present an algorithm to obtain a congruence subgroup $\varGamma_{H}$, on which there exists a weight $2$ eigenform $f_2$ such that $\rho_{f,\ell}$ is isomorphic to a twist of $\rho_{f_2,\ell}$. 

Moreover, we will prove the group $\varGamma_{H}$ produced by this algorithm is  the largest possible congruence subgroup with $\varGamma_{1}(N')\subseteq\varGamma_{H}\subseteq\varGamma_{0}(N')$, on which such eigenform $f_2$ exists. Here $N'$ is equal to $N\ell$ if the weight $k$ of $f$ is greater than $2$ and $N'=N$ if $k=2$.

First we state the following result without proof, which has been obtained independently by H. Carayol and J-P. Serre, and is usually called Carayol's Lemma.

\begin{theorem}[Carayol's Lemma] \label{carayol'slemma}
	Let $\ell\ge5$ be a prime and $v$ be a place dividing $\ell$ of $\overline{\mathbb{Q}}$. Let $f\in S_{k}(\varGamma_{1}(N),\varepsilon)$ be a normalized eigenform. Suppose the representation $\rho_{f,\ell}$ is irreducible.  Let $\varepsilon'$ be a Dirichlet character which is congruent to $\varepsilon \mod v$. Then there exists a normalized eigenform $f'\in S_{k}(\varGamma_{1}(N),\varepsilon')$ such that $\rho_{ f,\ell}$ and $\rho_{f',\ell}$ are isomorphic.
\end{theorem}
\begin{proof}
	See \cite[Proposition 3]{carayol}.
\end{proof}

Now we can show 
\begin{theorem}   \label{thm:mainforM}
	Let $\ell\ge 5$ be a prime number, $N>0$ an integer prime to $\ell$, and $k> 2$. Let $f\in S_{k}(\varGamma_{1}(N), \varepsilon)$ be a normalized eigenform and $\lambda$ be a prime of $K_f$ lying over $\ell$. Suppose the representation $\rho_{f,\lambda}$ is irreducible. Then there exist integers $i$ with $0\le i \le \ell-1$ and $M$ with $M|N\ell$,  a newform $f_2\in S_2(\varGamma_{H}(M))$, and a prime $\lambda_2$ lying over $\ell$  in the field $K_{f_2}$, such that $\rho_{f,\lambda}$ is isomorphic to $\rho_{f_2,\lambda_2}\otimes \chi^{i}_{\ell}$. Here  $H=\{x \ (\mathrm{mod} \  M) \ | \  \mathrm{gcd}(x,N\ell)=1 \  \mathrm{with} \ 0<x<N\ell \  \mathrm{and} \ \varepsilon(x) x^{k-2-2i} \equiv1\mod \lambda \}$.
\end{theorem}
\begin{proof}
	 Let $v$  be a place dividing $\lambda$ of $\overline{\mathbb{Q}}$. By Theorem \ref{thm:twist},  there exist $i$ and $k'$ with $0\le i\le \ell-1, \ k'\le \ell+1$,  a newform $g\in S_{k'}(\varGamma_{1}(M_{1}),\varepsilon)$,  and a prime   $\mathfrak{l}$ of $K_g$ lying over $\ell$, such that $\rho_{f,\lambda}$ is isomorphic to $\rho_{g,\mathfrak{l}}\otimes \chi^{i}_{\ell}$. Then  by (\ref{charpol}) we have the equality in $\overline{\mathbb{F}}$:
	\begin{equation} \label{det1}
		\chi_{\ell}^{k-1}=\chi_{\ell}^{k'-1+2i}.
	\end{equation}
	It follows from  \cite[Proposition 9.3]{gross} that there exist a newform $g_{2}\in S_{2}(\varGamma_{1}(M_{1}\ell),\psi)$ for some integer $M_{1}|N$ and a prime $\mathfrak{l}_{2}|\ell$, such that $\rho_{g,\mathfrak{l}}$ is isomorphic to $\rho_{g_{2},\mathfrak{l}_{2}}$. Again by (\ref{charpol}) we have the equality in $\overline{\mathbb{F}}$:
	\begin{equation} \label{det2}
		\overline{\psi}_{ind}\chi_{\ell}=\bar\varepsilon\chi_{\ell}^{k'-1},
	\end{equation}
    where $\psi_{ind}$ is the induced character mod $N\ell$ by $\psi$ 	and	 the bar denotes reduction modulo $v$.
	Therefore we have that   $\rho_{f,\lambda}$ is isomorphic to $\rho_{g_2,\mathfrak{l}_2}\otimes \chi^{i}_{\ell}$ and it follows from  (\ref{det1}) and (\ref{det2}) that 
	\begin{equation} \label{det}
		\overline{\psi}_{ind}=\bar\varepsilon\chi_{\ell}^{k-2-2i}. 
	\end{equation}
	Let $\mathrm{T}(\overline{\psi})$ be a Teichm\"uller lifting of $\overline{\psi}$ as in Definition \ref{def:teichmullerlifting}. By Theorem \ref{carayol'slemma}, we have a normalized eigenform $f_{2}\in S_{2}(\varGamma_{1}(M_{1}\ell),\mathrm{T}(\overline{\psi}))$, and a prime $\lambda_2$ lying over $\ell$  in the field $K_{f_2}$, such that $\rho_{f_2,\lambda_2}$ is isomorphic to $\rho_{g_2,\mathfrak{l}_2}$.
	By \cite[Theorem 1.2]{ribetnebentypus}, we can take $f_2$ to be a newform in $S_{2}(\varGamma_{1}(M),\varepsilon_{2})$, where  $M|M_{1}\ell$ is an integer and  $\varepsilon_{2}$ is a Dirichlet character modulo $M$ that induces $\mathrm{T}(\overline{\psi})$. 
	Thus we obtain a newform  $f_{2}\in S_{2}(\varGamma_{1}(M),\varepsilon_{2})$ such that $\rho_{f,\lambda}$ is isomorphic to $\rho_{f_2,\lambda_2}\otimes \chi^{i}_{\ell}$.
	
	Now we prove that $f_{2}$ is a newform on $\varGamma_{H}(M)$. Since  $\varepsilon_{2}$  induces the Teichm\"uller lifting  $\mathrm{T}(\overline{\psi})$, it follows from (\ref{inducedkernal}) and (\ref{det}) that 
	$H= \{x \ (\mathrm{mod} \ M) \ | \  \mathrm{gcd}(x,N\ell)=1 \  \mathrm{with} \ 0<x<N\ell \  \mathrm{and} \ \varepsilon (x) x^{k-2-2i} \equiv1\mod \lambda\} = \pi_{N\ell,M}(\mathrm{ker}(\bar\varepsilon\chi_{\ell}^{k-2-2i})) = \pi_{N\ell,M}(\mathrm{ker}(\overline{\psi}_{ind})) =\pi_{M_{1}\ell,M}(\mathrm{ker}(\overline{\psi}))  =  \mathrm{ker}(\varepsilon_{2})$.
	
	Note $H$ is a normal subgroup of  $(\mathbb{Z}/M\mathbb{Z})^{\ast}$.  It is evident that ker$(\varphi_{M})\subseteq\varGamma_H(M)$.
	Moreover, for any $\gamma=\begin{pmatrix} a&b\\ c&d\end{pmatrix} \in \varGamma_{H}(M)$, we have that $\varphi_{M}(\gamma)\in H=$ ker$(\varepsilon_{2})$ and thus $f_2|_{2}\gamma=\varepsilon_{2}(\varphi_{M}(\gamma))\cdot f_2=f_2$, which implies $f_{2} \in S_{2}(\varGamma_{H}(M))$.
\end{proof}

In Theorem \ref{thm:mainforM}, the form $f_2$ is a newform, but its level involves a divisor $M$ of $N\ell$. Note that the form $f_2$ can naturally seen as a normalized eigenform which has level $N\ell$. In the following corollary, we give a method to compute the congruence subgroup $\varGamma_H$ of level $N\ell$ on which  $f_2$ is an eigenform, but not necessarily a newform.
\begin{theorem}   \label{thm:main}
	Let $\ell\ge 5$ be a prime number, $N>0$ an integer prime to $\ell$, and $k>2$. Let $f\in S_{k}(\varGamma_{1}(N), \varepsilon)$ be a normalized eigenform and $\lambda$ be a prime of $K_f$ lying over $\ell$. Suppose the representation $\rho_{f,\lambda}$ is irreducible. Then there exist an integer $i$ with $0\le i \le \ell-1$, a normalized eigenform $f_2\in S_2(\varGamma_{H})$, and a prime $\lambda_2$ lying over $\ell$  in the field $K_{f_2}$, such that $\rho_{f,\lambda}$ is isomorphic to $\rho_{f_2,\lambda_2}\otimes \chi^{i}_{\ell}$. Here   $H=\{x \  (\mathrm{mod} \ N\ell)  \ | \  \mathrm{gcd}(x,N\ell)=1 \  \mathrm{with} \ 0<x<N\ell \  \mathrm{and} \ \varepsilon(x) x^{k-2-2i} \equiv1\mod \lambda \}$ and $\varGamma_{H}=\varGamma_{H}(N\ell)$.
\end{theorem}
\begin{proof}
	 By Theorem \ref{thm:mainforM}, there exist integers $i$ with $0\le i \le \ell-1$ and $M$ with $M|N\ell$,  a newform $f_2\in S_2(\varGamma_{H'}(M))$, and a prime $\lambda_2$ lying over $\ell$  in the field $K_{f_2}$, such that $\rho_{f,\lambda}$ is isomorphic to $\rho_{f_2,\lambda_2}\otimes \chi^{i}_{\ell}$. Here   $H'=\{x \  (\mathrm{mod} \ M) \ | \  \mathrm{gcd}(x,N\ell)=1 \  \mathrm{with} \ 0<x<N\ell \  \mathrm{and} \ \varepsilon(x) x^{k-2-2i} \equiv1\mod \lambda \}$.
	 
	 For any  $\gamma = \begin{pmatrix}
	 	a & b\\
	 	c & d
	 \end{pmatrix} \in \varGamma_{H}=\varGamma_{H}(N\ell)$, since $M$ is a divisor of $N\ell$, we know that  $c\equiv 0 \pmod M$ and $d \pmod M \in H'$. Then we have $\varGamma_{H}\subseteq \varGamma_{H'}(M)$, which implies $f_{2} \in S_{2}(\varGamma_{H})$. 
	\end{proof}

Note that  in the proofs of Theorem \ref{thm:mainforM} and \ref{thm:main}, the integer $i$ is determined by Theorem \ref{thm:twist}. Consequently, in the case with $\ell\ge k-1$, Theorem \ref{thm:main} boils down to the following corollary.
\begin{corollary} \label{cor:mainforlargel}
	
		Let $\ell\ge 5$ be a prime number, $N>0$ an integer prime to $\ell$, and $k>2$. Let $f\in S_{k}(\varGamma_{1}(N), \varepsilon)$ be a normalized eigenform and $\lambda$ be a prime of $K_f$ lying over $\ell$. Suppose the representation $\rho_{f,\lambda}$ is irreducible and $\ell\ge k-1$. Then there exist  a normalized eigenform $f_2\in S_2(\varGamma_{H})$ and a prime $\lambda_2$ lying over $\ell$  in the field $K_{f_2}$, such that $\rho_{f,\lambda}$ is isomorphic to $\rho_{f_2,\lambda_2}$. Here  $H=\{x \ (\mathrm{mod} \ N\ell)  \ | \  \mathrm{gcd}(x,N\ell)=1 \  \mathrm{with} \ 0<x<N\ell \  \mathrm{and} \ \varepsilon(x) x^{k-2} \equiv1\mod \lambda \}$ and $\varGamma_{H}=\varGamma_{H}(N\ell)$.
\end{corollary}  
\begin{proof}
	If $\ell\ge k-1$, the integer $i$ in Theorem \ref{thm:twist} can be taken to be $0$. Therefore, in Theorem \ref{thm:main} we have $i=0$ and $H=\{x \ (\mathrm{mod} N\ell)  \ | \  \mathrm{gcd}(x,N\ell)=1 \  \mathrm{with} \ 0<x<N\ell \  \mathrm{and} \ \varepsilon(x) x^{k-2} \equiv1\mod \lambda \}$. 
\end{proof}

	The following theorem  shows that the congruence subgroup $\varGamma_{H}$ in Theorem \ref{thm:main} is in fact the largest possible congruence subgroup with $\varGamma_{1}(N\ell)\subseteq\varGamma_{H}\subseteq\varGamma_{0}(N\ell)$, on which such eigenform $f_2$ exists.

\begin{theorem}   \label{thm:largest} 
	Let $\ell\ge 5$ be a prime number, $N>0$ an integer prime to $\ell$, and $k> 2$. Let $f\in S_{k}(\varGamma_{1}(N), \varepsilon)$ be a normalized eigenform  and $\lambda$ be a prime of $K_f$ lying over $\ell$. Suppose the representation $\rho_{f,\lambda}$ is irreducible. Suppose we have a normalized eigenform $g_2\in S_2(\varGamma)$  with $\rho_{f,\ell}\cong\rho_{g_2,\ell}\otimes \chi^{i}_{\ell}$ for some integer $i$ with $0\le i \le \ell-1$, and congruence subgroup $\varGamma$ with $\varGamma_{1}(N\ell)\subseteq\varGamma\subseteq\varGamma_{0}(N\ell)$. Let $H=\{x \ (\mathrm{mod} N\ell)  \ | \  \mathrm{gcd}(x,N\ell)=1 \  \mathrm{with} \ 0<x<N\ell \  \mathrm{and} \ \varepsilon(x) x^{k-2-2i} \equiv1\mod \lambda \}$ and $\varGamma_{H}=\varGamma_{H}(N\ell)$. Then we have $\varGamma\subseteq\varGamma_{H}$.
	
	Moreover, there exists a normalized eigenform $f_{2}\in S_2(\varGamma_H)$ such that $\rho_{f,\ell}$ is isomorphic to $\rho_{f_2,\ell}\otimes \chi^{i}_{\ell}$.
\end{theorem}
\begin{proof}
     Since  $g_2\in S_2(\varGamma)$ and  $\varGamma_{1}(N\ell) \subseteq \varGamma$, the form $g_2$ can be naturally seen as a form on $\varGamma_{1}(N\ell)$ with a  modulo $N\ell$ nebentypus character $\psi$.

    Let $\varphi_{N\ell}$ denote the surjection:
    \begin{displaymath}
    \varphi_{N\ell} :  \ \varGamma_{0}(N\ell) \twoheadrightarrow (\mathbb{Z}/N\ell\mathbb{Z})^{\ast} , \ \ \ \ \ \ \
    \begin{pmatrix} a & b\\ c & d \end{pmatrix} \rightarrow d \ (\mathrm{mod} \ N\ell).
    \end{displaymath}
     For any  $\gamma  \in \varGamma\subseteq\varGamma_{0}(N\ell)$, we have that $g_2=g_2|_{2}\gamma=\psi_{}(\varphi_{N\ell}(\gamma))\cdot g_2$, which implies that $\varphi_{N\ell}(\gamma)\in\mathrm{ker}(\psi_{})$.
    
    Since $\rho_{f,\ell}\cong \rho_{g_2,\ell}\otimes \chi^{i}_{\ell}$, by (\ref{charpol}) we have the equality in $\overline{\mathbb{F}}$:
   	\begin{displaymath} 
   \overline{\psi}_{}=\bar\varepsilon\chi_{\ell}^{k-2-2i}. 
   \end{displaymath}
    Note  $H$ actually is the kernel of $\bar\varepsilon\chi_{\ell}^{k-2-2i}$.
    It follows that $\varphi_{N\ell}(\gamma)\in\mathrm{ker}(\psi_{})\subseteq\mathrm{ker}(\overline{\psi}_{})=\mathrm{ker}(\bar\varepsilon\chi_{\ell}^{k-2-2i})= H$.  By the definition of $\varGamma_{H}=\varGamma_{H}(N\ell)$, we have $\gamma\in \varGamma_{H}$, and therefore $\varGamma \subseteq \varGamma_{H}$.
    
    Let $\varepsilon_{2}$ be a Teichm\"uller lifting of $\overline{\psi}_{}$ as in Definition \ref{def:teichmullerlifting}. By Theorem \ref{carayol'slemma}, we have a normalized eigenform $f_{2}\in S_{2}(\varGamma_{1}(N\ell),\varepsilon_{2})$ such that $\rho_{f_2,\ell}$ is isomorphic to $\rho_{g_2,\ell}$. Then we know that $\rho_{f,\ell}$ is isomorphic to $\rho_{f_2,\ell}\otimes \chi^{i}_{\ell}$ and it follows 
    \begin{equation} \label{det3}
    \overline{\varepsilon}_{2} =\bar\varepsilon\chi_{\ell}^{k-2-2i}. 
    \end{equation}
    
    Now we show  $f_{2}\in S_2(\varGamma_H)$.
    Since  $\varepsilon_{2}$ is a Teichm\"uller lifting of $\overline{\psi}_{}$, it follows from  (\ref{det3}) that  ker$(\varepsilon_{2})=\mathrm{ker}(\bar\varepsilon_{2})= \mathrm{ker}(\bar\varepsilon\chi_{\ell}^{k-2-2i})=H$.  Then for any $\gamma=\begin{pmatrix} a&b\\ c&d\end{pmatrix} \in \varGamma_{H}$, we have  $\varphi_{N\ell}(\gamma)\in\mathrm{ker}(\varepsilon_{2})$ and thus $f_2|_{2}\gamma=\varepsilon_{2}(\varphi_{N\ell}(\gamma))\cdot f_2=f_2$, which implies $f_{2}\in S_2(\varGamma_H)$.
\end{proof}

If the form $f$ has weight $2$, we have the following results.
    
\begin{theorem} \label{thm:mainfork=2}   
	Let $\ell\ge 5$ be a prime number and $N>0$ an integer prime to $\ell$. Let $f\in S_2(\varGamma_1(N),\varepsilon)$ be a normalized eigenform and $\lambda$ be a prime of $K_f$ lying over $\ell$. Suppose the modular Galois representation $\rho_{f,\lambda}$ is irreducible. Let   $H=\mathrm{ker}(\bar\varepsilon)$ be the kernel of the reduction of $\varepsilon$ modulo $\lambda$ and   $\varGamma_{H}=\varGamma_{H}(N)$. Then there exists a normalized eigenform $f_2 \in S_2(\varGamma_H)$ such that $\rho_{f,\ell}\cong \rho_{f_2,\ell}$.
	
	Moreover,  the group  $\varGamma_{H}$ is the largest possible congruence subgroup with $\varGamma_{1}(N)\subseteq\varGamma_{H}\subseteq\varGamma_{0}(N)$, on which such  eigenform $f_2$ exists. 
\end{theorem}
\begin{proof}
	We take $\varepsilon'$ to be a Teichm\"uller lifting of $\bar\varepsilon$, and the existence of $f_2$ just follows from Theorem \ref{carayol'slemma}. 
	
	 Let $g_2\in S_2(\varGamma)$ be a normalized eigenform, such that $\rho_{f,\ell}\cong\rho_{g_2,\ell}$ for some congruence subgroup $\varGamma$  with $\varGamma_{1}(N)\subseteq\varGamma\subseteq\varGamma_{0}(N)$.  We will show $\varGamma \subseteq \varGamma_{H}(N)$.
	 
	 Let $\psi$ be the nebentypus character of $g_2$.  Let $\varphi_{N}$ denote the surjection:
	 \begin{displaymath}
	 \varphi_{N} :  \ \varGamma_{0}(N) \twoheadrightarrow (\mathbb{Z}/N\mathbb{Z})^{\ast} , \ \ \ \ \ \ \
	 \begin{pmatrix} a & b\\ c & d \end{pmatrix} \rightarrow d \ (\mathrm{mod} \ N).
	 \end{displaymath}
	 For any  $\gamma  \in \varGamma$, we have that $g_2=g_2|_{2}\gamma=\psi(\varphi_{N}(\gamma))\cdot g_2$ and hence $\varphi_{N}(\gamma)\in\mathrm{ker}(\psi)$. 
	 Since $\rho_{f,\ell}\cong\rho_{g_2,\ell}$, by (\ref{charpol}) we have  $\overline{\psi}=\bar\varepsilon$. It follows that $\varphi_{N}(\gamma)\in\mathrm{ker}(\psi)\subseteq\mathrm{ker}(\overline{\psi})=\mathrm{ker}(\bar\varepsilon)= H$.  By the definition of $\varGamma_{H}(N)$, we have $\gamma\in \varGamma_{H}(N)$, and therefore $\varGamma \subseteq \varGamma_{H}(N)$.
\end{proof}

 If we suppose  $f\in S_2(\varGamma_1(N),\varepsilon)$ and gcd$(\ell, \phi(N))=1$, by Lemma \ref{lem:teichmullerliftingforlandphicoprime},  we have  ker$(\bar\varepsilon)=\mathrm{ker}(\varepsilon)$. Therefore  the group $H$ in Theorem \ref{thm:mainfork=2}, which is the kernel  of the reduction  $\bar\varepsilon$, is also the kernel of $\varepsilon$. Then we have
\begin{corollary}  
	Let $\ell\ge 5$ be a prime number and $N>0$ an integer prime to $\ell$. Let $f\in S_2(\varGamma_1(N),\varepsilon)$ be a normalized eigenform. Let $H=\mathrm{ker}(\varepsilon)$ be the kernel of $\varepsilon$  and $\varGamma_{H}=\varGamma_{H}(N)$.  Suppose gcd$(\ell, \phi(N))=1$ and the modular Galois representation $\rho_{f,\lambda}$ is irreducible. Then there exists a normalized eigenform $f_2 \in S_2(\varGamma_H)$ such that $\rho_{f,\ell}\cong \rho_{f_2,\ell}$.
	
	Moreover, the group $\varGamma_{H}$ is the largest possible congruence subgroup with $\varGamma_{1}(N)\subseteq\varGamma_{H}\subseteq\varGamma_{0}(N)$, on which such  eigenform $f_2$ exists. 
\end{corollary}
\begin{proof}
	It just follows from Theorem \ref{thm:mainfork=2} and Lemma \ref{lem:teichmullerliftingforlandphicoprime}.
\end{proof}

Then we have the following algorithm.

\begin{algorithm}  \label{alg:main}
	Let $\ell\ge 5$ be a prime number, $N>0$ an integer prime to $\ell$, and $k\ge2$. Let $f\in S_{k}(\varGamma_{1}(N), \varepsilon)$ be a normalized eigenform and $\lambda$ be a prime of $K_f$ lying over $\ell$. Suppose the representation $\rho_{f,\lambda}$ is irreducible.  Let $N'= N$ if $k=2$ and  $N' = N\ell$ if $k>2$.  
	 This algorithm outputs an integer $i$ with $0\le i \le \ell-1$, a normalized eigenform $f_2\in S_2(\varGamma_{H})$, and a prime $\lambda_2$ lying over $\ell$  in the field $K_{f_2}$, such that $\rho_{f,\lambda}$ is isomorphic to $\rho_{f_2,\lambda_2}\otimes \chi^{i}_{\ell}$. 
	 Here $H=\{x \ (\mathrm{mod} \ N')  \ | \ \mathrm{gcd}(x,N')=1 \ \mathrm{with} \ 0<x<N' \ \mathrm{and} \ \varepsilon(x) x^{k-2-2i} \equiv1\mod \lambda \}$  and $\varGamma_{H}=\varGamma_{H}(N')$.
	\begin{enumerate}[$1.$]
		\item    Set $i\leftarrow0$ if $k=2$ or $\ell\ge k-1$. Otherwise compute  $i$ by Algorithm \ref{alg:twistforsmalll}.
		\item Compute $M$ by Algorithm \ref{alg:twistforsmalll}.
		\item Set $M'\leftarrow N$ if $k=2$ and  $M' \leftarrow M\ell$ if $k>2$.  
		\item  Compute the set $S$ consisting of all the divisors of $M'$.
		\item  Take $M''$ in $S$ and do:
		\begin{enumerate}
			\item 	 Compute the group $H'=\{x \ (\mathrm{mod} \  M'') \ | \  \mathrm{gcd}(x,N'\ell)=1 \  \mathrm{with} \ 0<x<N'\ell \  \mathrm{and} \ \varepsilon(x) x^{k-2-2i} \equiv1\mod \lambda \}$.
			\item   Compute the congruence subgroup
			$\varGamma_{H'}(M'')$
			\item Compute $B=\frac{[SL_{2}(\mathbb{Z}):\varGamma_{1}(M'')]}{12}\cdot(\ell^2-1+k)$ and  $a_{p}(f)$ for all primes $p$ with  $p \le B$.
			\item   Compute all  newforms $F$ in  $S_{2}(\varGamma_{H'}(M''))$ using modular symbols. 
			\item   For each $f_{2}$ in $F$, do:
			\begin{enumerate}
				\item Compute  $p^{i}a_{p}(f_{2})$  for all primes $p$ with  $p\le B$ and compute primes $P$ of the composed field $K_{f}K_{g}$ lying over $\ell$. 
				\item 	         If there is a prime  $\mathfrak{l} \in P$ such that  $a_{p}(f) \equiv p^{i}a_{p}(f_2) \mod \mathfrak{l}$ for all primes $p$ with  $p    \le B$,  put $\lambda_{2}=\mathfrak{l}\cap K_{g}$ and then output $i$,  $f_{2}$, $\lambda_{2}$, and terminate. 
			\end{enumerate}
			\item Set $S\leftarrow S-\{M\}$ and go to step $5$.  
		\end{enumerate}
	\end{enumerate}
\end{algorithm}

\subsection{To realize modular Galois representations in Jacobians of the smallest possible dimensions}
	 Let $N>0$ and $k\ge2$ be integers. Let $f\in S_{k}(\varGamma_{1}(N),\varepsilon)$ be a normalized eigenform.  Let $\ell$ be a prime number with $\ell \nmid N$ and $\lambda$ be a prime of $K_f$ lying over $\ell$.  Let $N'= N$ if $k=2$ and  $N' = N\ell$ if $k>2$. Suppose the representation $\rho_{f,\lambda}$ is irreducible. Then by Algorithm \ref{alg:main}, we can obtain  an integer $i$ with $0\le i \le \ell-1$,  a normalized eigenform $f_2\in S_2(\varGamma_{H},\varepsilon_{2})$, and a prime $\lambda_2$ lying over $\ell$  in the field $K_{f_2}$, such that $\rho_{f,\lambda}$ is isomorphic to $\rho_{f_2,\lambda_2}\otimes \chi^{i}_{\ell}$. Here $H=\{x \ | \ \mathrm{gcd}(x,N')=1 \ \mathrm{with} \ 0<x<N' \ \mathrm{and} \ \varepsilon(x) x^{k-2-2i} \equiv1\mod \lambda \}$  and $\varGamma_{H}=\varGamma_{H}(N')$.  
	
Now we return to  the computations of $\rho_{f,\lambda}$.  From the discussion of Section \ref{sec:preliminary}, we know it suffices to compute the representation 
$$
\rho_{V_{f_{2}}}: Gal (\mathbb{\overline{Q}}/\mathbb{Q}) \rightarrow \mathrm{Aut}_{\mathbb{T}/\mathfrak{m}_{f_{2}}}(V_{f_{2}}),
$$
where $V_{f_{2}}=J_{1}(N')(\overline{\mathbb{Q}})[\mathfrak{m}_{f_{2}}]=\{x \in J_{1}(N')(\overline{\mathbb{Q}}) \ | \ tx=0 \ \mathrm{for}  \ \mathrm{all} \ t \ \mathrm{in} \ \mathfrak {m}_{f_2} \}$, since $\rho_{f_2}$ is isomorphic to $\rho_{V_{f_{2}}}$ or any simple constituent of $\rho_{V_{f_{2}}}$.

Let $X_{\varGamma_{H}}$ be the modular curve of the subgroup $\varGamma_{H}$ and denote $J_{\varGamma_{H}}$ its Jacobian.
By the Galois theory of function fields of modular curve, we know that the holomorphic differential space $\varOmega^{1}_{hol}(X_{\varGamma_{H}})$ is the $H$-invariant part of the space $\varOmega^{1}_{hol}(X_{1}(N'))$. By taking duals of the two spaces, we have 
\begin{displaymath}
J_{\varGamma_{H}}=\{x \in J_{1}(N')  \ | \ \sigma(x)=x, \mathrm{ \ for \ all \ } \sigma \in H \}.
\end{displaymath} 
Then we can show
\begin{theorem} \label{thm:jacobian}
 The torsion space $V_{f_2}$ is a $2$-dimensional subspace of $J_{\varGamma_{H}}[\ell]$. Therefore, the representation $\rho_{V_{f_{2}}}$ is a $2$-dimensional subrepresentation of $J_{\varGamma_{H}}[\ell]$.
\end{theorem}
\begin{proof}
 Since $H$ is in fact the normal subgroup ker$(\varepsilon_{2})$ of $(\mathbb{Z}/(N')\mathbb{Z})^{\ast}$, it follows that the action of each $\sigma\in H$ on the $\ell$-torsion points $J_{1}(N')[\ell]$ of  $J_{1}(N')$ is the same as the action of a diamond operator $\langle d \rangle$ on $J_{1}(N')[\ell]$ for some $d\in (\mathbb{Z}/N'\mathbb{Z})^{\ast}$ with $\varepsilon_{2}(d) \equiv1\mod \lambda_{2}$. Let  $\mathfrak{m}_{f_{2}}$ be the kernel of the homomorphism $$\mathcal{I}_{f_2}: \mathbb{T} \rightarrow \mathbb{\overline F_{\ell}}, \ \  \ \ \langle d \rangle \mapsto \varepsilon(d), \ \  T_{n} \mapsto a_{n}(f) \mod \lambda_{2}.$$
 Then we have $\sigma-id$ is an element of $\mathfrak{m}_{f_{2}}$ and therefore we have $V_{f_{2}}\subseteq J_{\varGamma_{H}}[\ell]$. 
\end{proof}

By the argument at the end of Section \ref{sec:preliminary}, we know that the calculations can be largely decreased if we can realize the modular Galois representation $\rho_{V_{f_{2}}}$ in a Jacobian which has a smaller dimension. Theorem \ref{thm:jacobian} allows us to work with  $J_{\varGamma_{H}}$ instead of $J_{1}(N')$ to compute $\rho_{V_{f_{2}}}$. Since the dimension of $J_{\varGamma_{H}}$ is the same with the dimension of the $\mathbb{C}$-vector space $S_{2}(\varGamma_{H})$,  it follows from Theorem \ref{thm:largest} that the Jacobian $J_{\varGamma_{H}}$ found by our method has the smallest possible dimension, in the sense that $\varGamma_{H}$ is the largest possible congruence subgroup  with $\varGamma_{1}(N')\subseteq\varGamma_{H}\subseteq\varGamma_{0}(N')$ associated to the representation  $\rho_{f,\lambda}$.  

Given $f\in S_{12}(\varGamma_{1}(N))$, in Table $1$ to $5$,  we show the eigenforms $f_2$ produced by Algorithm \ref{alg:main} in the cases with $\ell$ up to $13$ and $N$ up to $6$. We also list the dimensions of $J_{1}(N')$ and  $J_{\varGamma_{H}}$  which are denoted by  $d_1$ and $d_{H}$, respectively.

\vspace{1cm}

        \makeatletter\def\@captype{table}\makeatother

	\caption{$N=1$}\label{eqtable}
	\renewcommand\arraystretch{1.5}
	\noindent\[
	\begin{array}{|c|c|c|c|c|c|c|c|}
	
	\multicolumn{6}{c}{f= q - 24q^2 + 252q^3 + O(q^4) $\ and \ $ K_f= \mathbb{Q}} \\

	\hline
	\ell & \lambda & \lambda_{2} & i  & f_2 & K_{f_2} & d_1 & d_{H} \\
	\hline
	11& (11) & (11) & 1  & q - 2q^2 - q^3  + O(q^4)  &  \mathbb{Q} &1&1\\
	\hline
	13 & (13) &(-4\alpha - 7) & 0 &  \begin{minipage}[t]{0.35\textwidth} $q + \alpha\cdot q^2 + (-2\alpha - 4)\cdot q^3 + O(q^4)$ \end{minipage} &  x^2 + 3x + 3 &2&2\\
	\hline
	
	\end{array}
	\]

\vspace{1cm}

        \makeatletter\def\@captype{table}\makeatother

	\caption{$N=3$}\label{eqtable}
	\renewcommand\arraystretch{1.5}
	\noindent\[
	\begin{array}{|c|c|c|c|c|c|c|c|}

	\multicolumn{6}{c}{f=q + 78q^2 - 243q^3  + O(q^4) $\ and \ $ K_f= \mathbb{Q}} \\

	\hline
	\ell & \lambda & \lambda_{2} & i  & f_2  & K_{f_2}  & d_1 & d_{H} \\
	
	\hline
	5 & (5) &(5) & 1  & q - q^2 - q^3  + O(q^4)  &  \mathbb{Q}  & 1 & 1 \\
	\hline
	7  & (7) & (\frac{3}{2}\alpha + 1)
	 & 0 & \begin{minipage}[t]{0.37\textwidth} $q + \alpha \cdot q^2 + (-\frac{1}{2}\alpha - 1)\cdot q^3 + O(q^4)$ \end{minipage} & x^2 + 2x + 4 & 5 & 3 \\
	\hline
	11& (11) & (11) & 0 & q + q^2 - q^3 + O(q^4)  &  \mathbb{Q}  & 21 & 3 \\
	\hline
	13 & (13) &(4\alpha + 1) & 0  & q - \alpha  \cdot q^3   + O(q^4) & x^2 + x + 1 & 33 & 17 \\
	\hline
	
	\end{array}
	\]

\newpage
        \makeatletter\def\@captype{table}\makeatother

	\caption{$N=4$}\label{eqtable}
	\renewcommand\arraystretch{1.5}
	\noindent\[
	\begin{array}{|c|c|c|c|c|c|c|c|}

\multicolumn{6}{c}{f= q - 516q^3  + O(q^4) $\ and \ $ K_f= \mathbb{Q}} \\

	\hline
	\ell & \lambda & \lambda_{2} & i  & f_2 & K_{f_2}  & d_1 & d_{H} \\
    \hline
	
	5  & (5) &(5) & 1  &  q - 2q^3 + O(q^4) &  \mathbb{Q} & 3 & 1 \\
	\hline
	7 & (7) & (\frac{3}{2}\alpha - 2)  & 0  &  q - \frac{1}{2}\alpha \cdot q^3 + O(q^4) & x^2 - 2x + 4 & 10 & 4 \\
	\hline
	11 & (11) & (11) & 0  & q + q^3 + O(q^4) &  \mathbb{Q} & 36 & 4 \\
	\hline
	13  & (13) & (2\alpha + 1) & 0 & q - \frac{1}{2}\alpha\cdot q^3  + O(q^4)& x^2 + 2x + 4 & 55 & 25 \\
	\hline
	
	\end{array}
	\]

\vspace{1.5cm}
        \makeatletter\def\@captype{table}\makeatother

	\caption{$N=5$}\label{eqtable}
	\renewcommand\arraystretch{1.5}
	\noindent\[
	\begin{array}{|c|c|c|c|c|c|c|c|}
	
		\multicolumn{6}{c}{f=q + a \cdot q^2 + (-\frac{1}{112}a^3 - \frac{725}{28}a) \cdot q^3 + O(q^4) $\ and $}\\ 
		
		\multicolumn{6}{c}{ K_f $ \ is \ the \ number \ field \ defined \ by \ $ x^4 + 4132x^2 + 2496256} \\
	
		\hline
	\ell & \lambda & \lambda_{2} & i & f_2 & K_{f_2}   & d_1 & d_{H} \\
	
	\hline
	7  &
	\begin{minipage}[t]{0.12\textwidth} $(7, \frac{1}{60}a^2 + \frac{494}{15})$ \end{minipage}  & \begin{minipage}[t]{0.19\textwidth} $(7, \frac{-a\cdot \alpha^3}{4} + (\frac{-a^2}{120} - \frac{217}{15})\cdot \alpha^2 + \frac{5a\cdot\alpha}{4} - \frac{a^2}{60} - \frac{524}{15})$  \end{minipage}
	 & 0 &  \begin{minipage}[t]{0.09\textwidth}
	$q + \alpha\cdot q^2 + (\alpha^3 - \alpha)\cdot q^3  + O(q^4)$
	\end{minipage} & \begin{minipage}[t]{0.12\textwidth} $x^4 - x^2 + 1$ \end{minipage} & 25 & 13 \\
	\hline
	11 &\begin{minipage}[t]{0.12\textwidth} $(11, \frac{1}{60}a^2 + \frac{434}{15})$ \end{minipage} &  \begin{minipage}[t]{0.22\textwidth}
	$(11, (\frac{a^3}{49280} - \frac{a^2}{120} - \frac{1711\cdot a}{12320} - \frac{1033}{60})\cdot\alpha^3 + (\frac{-a^2}{440} - \frac{599}{110})\cdot\alpha^2 + (\frac{-13a^3}{49280} - \frac{a^2}{24} - \frac{27037\cdot a}{12320})\cdot\alpha + \frac{a^3}{2240} + \frac{a^2}{220} + \frac{1369a}{560} + \frac{373}{110})
	$
	\end{minipage} & 0  & \begin{minipage}[t]{0.08\textwidth}
	$ q + \alpha \cdot q^2 + (-\frac{1}{2}\alpha^3 - \frac{7}{2}\alpha)\cdot q^3 + O(q^4)$
	\end{minipage}  &  x^4 + 7x^2 + 4 & 81 & 9 \\
	\hline
	13 &\begin{minipage}[t]{0.16\textwidth}
	$(13, -\frac{1}{3360}a^3 + \frac{1}{30}a^2 - \frac{809}{840}a + \frac{1048}{15})$
	\end{minipage}   &   \begin{minipage}[t]{0.21\textwidth}
	$(13, (\frac{13a^3}{26880} + \frac{10517a}{6720} - \frac{13}{8})\cdot \alpha^7 + (\frac{13}{5040a^3} + \frac{10517a}{1260} - \frac{26}{3})\cdot\alpha^5 + (\frac{247a^3}{20160} + \frac{199823a}{5040} - \frac{247}{6})\cdot\alpha^3 + \alpha^2 + (\frac{403a^3}{80640} + \frac{285707a}{20160} - \frac{355}{24})\cdot\alpha - \frac{a^2}{12} + \frac{5a}{4} - \frac{506}{3})$
	\end{minipage} & 0  & \begin{minipage}[t]{0.1\textwidth}
	$q + \alpha \cdot q^2 + (-\frac{13}{24}\alpha^7 - \frac{8}{3}\alpha^5 - \frac{38}{3}\alpha^3 - \frac{11}{8}\alpha)\cdot q^3  + O(q^4)$
	\end{minipage} & \begin{minipage}[t]{0.12\textwidth}
	$x^8 + 5x^6 + 24x^4 + 5x^2 + 1$
	\end{minipage} & 121 & 25 \\
	\hline
	
	\end{array}
	\]

\newpage
     \makeatletter\def\@captype{table}\makeatother

	\caption{$N=6$}\label{eqtable}
	\renewcommand\arraystretch{1.5}
	\noindent\[
	\begin{array}{|c|c|c|c|c|c|c|c|}
	
	\multicolumn{6}{c}{f= q - 32q^2 - 243q^3  + O(q^4)$\ and \ $ K_f= \mathbb{Q}} \\

	\hline
	\ell & \lambda & \lambda_{2} & i  & f_2 & K_{f_2}  & d_1 & d_{H} \\
	
	\hline
	5 & (5) &(2\alpha+1) & 0 & \begin{minipage}[t]{0.29\textwidth}  $q + \alpha \cdot q^2 - \alpha\cdot q^3 + O(q^4)$ \end{minipage} &  x^2+1  & 9 & 5 \\
	\hline
	7 & (7) & (3\alpha - 2)  & 4 &  \begin{minipage}[t]{0.29\textwidth}  $q + \alpha \cdot q^2 + (\alpha - 1)\cdot q^3  + O(q^4)$  \end{minipage} & x^2 - x + 1  & 25 & 13  \\
	\hline
	11& (11) & (11) & 0 & \begin{minipage}[t]{0.29\textwidth}  $q + q^2 - q^3  + O(q^4)$  \end{minipage} &  \mathbb{Q}  & 81 & 9 \\
	\hline
	13 & (13) & (-\alpha^3 - \alpha - 1) & 0 &  \begin{minipage}[t]{0.29\textwidth}  $q + \alpha\cdot q^2 + (1-\alpha^2)\cdot q^3  + O(q^4)$  \end{minipage}  & x^4 - x^2 + 1  & 121 & 61 \\
	\hline
	
	\end{array}
	\]

\vspace{0.5cm}

\section{Reduction to the cases of eigenforms on $\varGamma_{0}$}  \label{sec:gamma0}

In this section, we discuss the case that $k>2$ and $f\in S_{k}(\varGamma_{0}(N))$ is an eigenform on $\varGamma_{0}(N)$.

Now let $\phi(n)$ be the Euler's totient function. We first show the following lemma.

\begin{lemma}   \label{thm:proporderofkernel}
	Let $k\ge0$  and $m>0$ be integers, and $\ell$  a prime factor of $m$. Then the kernel of the homomorphism
	\begin{displaymath}
		\vartheta: \ (\mathbb{Z}/m\mathbb{Z})^{\ast}\rightarrow (\mathbb{Z}/\ell\mathbb{Z})^{\ast} , \ \ \ \ \ \ \
		x \bmod m  \ \rightarrow \ \ \ \  x^k\bmod \ell
	\end{displaymath} 
	has order $\frac{\phi(m)\cdot \mathrm{gcd}(\ell-1,k)}{\ell-1}$.
\end{lemma}
\begin{proof} 
	Since $\ell$ is a prime factor of $m$, the homomorphism $\vartheta$ factors as:
	\begin{displaymath} 
		\xymatrix{(\mathbb{Z}/m\mathbb{Z})^{\ast} \ar[rr]^{\vartheta}\ar[dr]_{\alpha}  &   & (\mathbb{Z}/\ell\mathbb{Z})^{\ast}  \\
			&  (\mathbb{Z}/\ell\mathbb{Z})^{\ast} \ar[ur]^{\beta} &
		}
	\end{displaymath}
	where $\alpha$ is the canonical homomorphism   
	\begin{displaymath}
		(\mathbb{Z}/m\mathbb{Z})^{\ast}\rightarrow (\mathbb{Z}/\ell\mathbb{Z})^{\ast} , \ \ \ \ \ \ \
		x \mod m  \ \rightarrow \ \ \ \  x\mod \ell,
	\end{displaymath} 	
	and $\beta$ is the homomorphism
	\begin{displaymath}
		(\mathbb{Z}/\ell\mathbb{Z})^{\ast} \rightarrow (\mathbb{Z}/\ell\mathbb{Z})^{\ast} , \ \ \ \ \ \ \
		x \mod \ell  \ \rightarrow \ \ \ \  x^k\mod \ell.
	\end{displaymath} 
	
	From Lemma \ref{lemmaofnd}, we know $\alpha$ is surjective and therefore the image  Im$(\vartheta)$ of $\vartheta$ is the same with the image  Im$(\beta)$ of $\beta$. Let $g$ be a generator of the cyclic group $(\mathbb{Z}/m\mathbb{Z})^{\ast}$ and then we know it has order $\ell-1$. It follows that Im$(\beta)=<g^k>$ has order $\frac{\ell-1}{\mathrm{gcd}(\ell-1,k)}$, which implies that the order of Im$(\vartheta)$ is also equal to $\frac{\ell-1}{\mathrm{gcd}(\ell-1,k)}$. 
	
	Since  $(\mathbb{Z}/m\mathbb{Z})^{\ast}/\mathrm{ker}
	(\vartheta)\cong \mathrm{Im}\vartheta$ and  $(\mathbb{Z}/m\mathbb{Z})^{\ast}$ has order $\phi(m)$, it follows that the kernel of $\vartheta$	has order $\frac{\phi(m)\cdot \mathrm{gcd}(\ell-1,k)}{\ell-1}$.
\end{proof}

Then we can show

\begin{theorem}  \label{thm:maingamma0}   
	Let $\ell\ge 5$ be a prime number, $N>0$ an integer prime to $\ell$, and $k>2$. Let $f\in S_{k}(\varGamma_{0}(N))$ be a normalized eigenform and $\lambda$ be a prime of $K_f$ lying over $\ell$. Suppose the representation $\rho_{f,\lambda}$ is irreducible. Let $i$ be the integer with $0\le i \le \ell-1$ and $\varGamma_{H}$ be the congruence subgroup  as given in Theorem \ref{thm:main}. Then the index $[\varGamma_{H}:\varGamma_1(N\ell)]$ of $\varGamma_{1}(N\ell)$  in $\varGamma_H$ is $\frac{\phi(N\ell)\cdot \mathrm{gcd}(\ell-1,k-2-2i)}{\ell-1}$. 
\end{theorem}
\begin{proof} 
	By Theorem \ref{thm:main}, there exists  a normalized eigenform $f_2\in S_2(\varGamma_{H},\varepsilon_{2})$, such that $\rho_{f,\ell}$ is isomorphic to $\rho_{f_2,\ell}\otimes \chi^{i}_{\ell}$. Here $H=\{x \  (\mathrm{mod}N\ell)   \ | \ \mathrm{gcd}(x,N\ell)=1 \ \mathrm{with} \ 0<x<N\ell \ \mathrm{and} \ \varepsilon(x) x^{k-2-2i} \equiv1\mod \lambda \}$  and $\varGamma_{H}=\varGamma_{H}(N\ell)$. 
     
     Since the nebentypus character of $f\in S_{k}(\varGamma_{0}(N))$ is trivial, it follows that $H=\{x \  (\mathrm{mod}N\ell)    \ | \  \mathrm{gcd}(x,N\ell)=1 \  \mathrm{with} \ 0<x<N\ell \  \mathrm{and} \  x^{k-2-2i} \equiv1\mod \ell \}$. Let $\vartheta$ be the homomorphism:
     	\begin{displaymath}
     \vartheta: \ (\mathbb{Z}/N\ell\mathbb{Z})^{\ast}\rightarrow (\mathbb{Z}/\ell\mathbb{Z})^{\ast} , \ \ \ \ \ \ \
     x \pmod{N\ell}  \ \rightarrow \ \ \ \  x^{k-2-2i}\pmod{\ell}.
     \end{displaymath} 
     Then it is evident that $H=\mathrm{ker}(\vartheta)$.
	 It follows from  Lemma \ref{thm:proporderofkernel} that $\#H= \frac{\phi(N\ell)\cdot \mathrm{gcd}(\ell-1,k-2-2i)}{\ell-1}$. 
	
	 Let $\varphi_{N\ell}$ denote the surjective homomorphism:
	\begin{displaymath}
	\varphi_{N\ell} :  \ \varGamma_{0}(N\ell) \twoheadrightarrow (\mathbb{Z}/N\ell\mathbb{Z})^{\ast} , \ \ \ \ \ \ \
	\begin{pmatrix} a & b\\ c & d \end{pmatrix} \rightarrow d \ (\mathrm{mod} \ N\ell).
	\end{displaymath}
	Then $\varGamma_{1}(N\ell)$ is the kernel of  $\varphi_{N\ell}$ and $\varGamma_{H}$ is the preimage $\varphi_{N\ell}^{-1}(H)$ of $H$ under $\varphi_{N\ell}$. It follows that $\varGamma_{H}/\varGamma_{1}(N\ell)\cong H$, and hence, the index $[\varGamma_{H}:\varGamma_1(N\ell)]=\# \ (\varGamma_{H}/\varGamma_{1}(N\ell))=\#H=\frac{\varphi(N\ell)\cdot \mathrm{gcd}(\ell-1,k-2-2i)}{\ell-1}$.
\end{proof}

If $f$ is an eigenform on $\varGamma_{0}(N)$, Theorem \ref{thm:maingamma0} implies the following corollary, which shows when the group $\varGamma_H$ must be $\varGamma_{0}(N\ell)$.

\begin{corollary}  \label{cor:gamma0ifandonlyifgamma0}
	Let $\ell\ge 5$ be a prime number, $N>0$ an integer prime to $\ell$, and $k>2$. Let $f\in S_{k}(\varGamma_{0}(N))$ be a normalized eigenform. Suppose the representation $\rho_{f,\ell}$ is irreducible. 	Let $i$ be the integer with $0\le i \le \ell-1$ and $\varGamma_{H}$ be the congruence subgroup  as given in Theorem \ref{thm:main}.  Then  $\varGamma_H=\varGamma_{0}(N\ell)$ if and only if $\ell-1 | k-2-2i$. 
\end{corollary}

\begin{proof}
	It follows from Theorem \ref{thm:maingamma0} that   $[\varGamma_H:\varGamma_1(N\ell)] = \frac{\phi(N\ell)\cdot \mathrm{gcd}(\ell-1,k-2-2i)}{\ell-1}$. Then $\varGamma_H=\varGamma_0(N\ell)$ if and only if $[\varGamma_H:\varGamma_1(N\ell)]=[\varGamma_{0}(N\ell):\varGamma_1(N\ell)]=\phi(N\ell)$, and hence if and only if $k-2-2i$ is divisible by $\ell-1$.
\end{proof}

    If we suppose $\ell\ge k-1$,  the integer $i$ as given in Corollary \ref{cor:gamma0ifandonlyifgamma0} can be taken to  be  $0$, and hence $\varGamma_H=\varGamma_{0}(N\ell)$ if and only $\ell=k-1$. Thus we can show 
\begin{corollary}  
		Let $\ell\ge 5$ be a prime number, $N>0$ an integer prime to $\ell$, and $k>2$. Let $f\in S_{k}(\varGamma_{0}(N))$ be a normalized eigenform.  Suppose $\ell\ge k-1$ and the representation $\rho_{f,\ell}$ is irreducible. Then there exists  a normalized eigenform $f_2\in S_2(\varGamma_{0}(N\ell))$  with $\rho_{f,\ell}\cong\rho_{f_2,\ell}$ if and only if $\ell=k-1$. 
\end{corollary}
\begin{proof}
	Let $\varGamma_{H}$ be the congruence subgroup as given in Corollary \ref{cor:mainforlargel}. By Theorem \ref{thm:largest}, we know that a normalized eigenform $f_2\in S_2(\varGamma_{0}(N\ell))$ with $\rho_{f,\ell}\cong\rho_{f_2,\ell}$ exists if and only if $\varGamma_{H}=\varGamma_{0}(N\ell)$. Since we can take $i$ to be $0$ in this case, this corollary just follows from  Corollary \ref{cor:gamma0ifandonlyifgamma0}.
\end{proof}

 For an eigenform $f\in S_{k}(\varGamma_{1}(N))$,  let $i$ be the integer with $0\le i \le \ell-1$ and $\varGamma_{H}$ be the congruence subgroup as  given in Theorem \ref{thm:main}. If suppose gcd$(\ell, \phi(N))=1$  and   $\ell-1 | k-2-2i$,
 we can show that the condition $\varGamma_H=\varGamma_{0}(N\ell)$ conversely implies that $f$ is an eigenform on $\varGamma_{0}(N)$. In fact,  in the following theorem, we will show that the form $f_2$  as given in Theorem \ref{thm:main} is a form on $\varGamma_{0}(N\ell)$ if and only if $f$ is a form on $\varGamma_{0}(N)$.
\begin{theorem}  \label{thm:gamma0iffgamma0}
	Let $\ell\ge 5$ be a prime number, $N>0$ an integer prime to $\ell$, and $k>2$. Let $f\in S_{k}(\varGamma_{1}(N))$ be a normalized eigenform. Suppose the representation $\rho_{f,\ell}$ is irreducible. Let $i$ be the integer with $0\le i \le \ell-1$ and $\varGamma_{H}$ be the congruence subgroup as  given in Theorem \ref{thm:main}. Suppose $\ell \nmid \phi(N)$ and   $\ell-1 | k-2-2i$. Then  $\varGamma_H=\varGamma_{0}(N\ell)$ if and only if $f\in S_{k}(\varGamma_{0}(N))$. 
\end{theorem}
\begin{proof}
	The sufficiency follows from the sufficiency of Corollary \ref{cor:gamma0ifandonlyifgamma0}. Now we prove the necessity.
	
	By Theorem \ref{thm:main}, there exists  a normalized eigenform $f_2\in S_2(\varGamma_{H},\varepsilon_{2})$, such that $\rho_{f,\ell}$ is isomorphic to $\rho_{f_2,\ell}\otimes \chi^{i}_{\ell}$. Let $\varepsilon $  be the  nebentypus characters of $f$. Then we have
	\begin{equation} \label{congruenceofgammaoiffgamma0}
		\bar\varepsilon_{2} \equiv \bar\varepsilon_{ind}\cdot \chi_{\ell}^{k-2-2i} \mod v,
	\end{equation}
	where 	$\varepsilon_{ind}$ is the mod $N\ell$ character induced by $\varepsilon$.
	
	Suppose  $\varGamma_H=\varGamma_{0}(N\ell)$ and then $\varepsilon_{2}$ is a trivial character. We also have  $\ell-1 | k-2-2i$, and it implies that the  congruence (\ref{congruenceofgammaoiffgamma0}) reduces to 
	\begin{displaymath} 
		\bar\varepsilon_{ind} \equiv 1 \mod v.
	\end{displaymath}
	Since $\varepsilon_{ind}=\varepsilon \circ \pi_{N\ell,N}$ and $\pi_{N\ell,N}$ is surjective by Lemma \ref{lemmaofnd}, we therefore have 
		\begin{displaymath} 
	\bar\varepsilon \equiv 1 \mod v.
	\end{displaymath}
	Since $\varepsilon$ is a Dirichlet character of $(\mathbb{Z}/N \mathbb{Z})^{\ast}$, each element of its image is a $\phi(N)$-th root of unity. We have $\ell \nmid \phi(N)$, and  it follows from Lemma \ref{lem:distinctrootsofunity} that the image of $\varepsilon$ does not contain any other  $\phi(N)$-th root of unity except $1$.  Hence $\varepsilon$ is the trivial character and this shows $f\in S_{k}(\varGamma_{0}(N))$.
\end{proof}

 If we suppose $\ell\ge k-1$, Theorem \ref{thm:gamma0iffgamma0} is reduced to the following corollary.
\begin{corollary}
	Let $\ell\ge 5$ be a prime number, $N>0$ an integer prime to $\ell$, and $k>2$. Let $f\in S_{\ell+1}(\varGamma_{1}(N))$ be a normalized eigenform.  Suppose $\ell \nmid \phi(N)$ and the representation $\rho_{f,\ell}$ is irreducible. Then there exists  a normalized eigenform $f_2\in S_2(\varGamma_{0}(N\ell))$  with $\rho_{f,\ell}\cong\rho_{f_2,\ell}$ if and only if  $f\in S_{\ell+1}(\varGamma_{0}(N))$. 
\end{corollary}
\begin{proof}
   Let $k=\ell+1$ denote the weight of $f$. Then we have  $\ell\ge k-1$ and $\ell-1 | k-2$.
    Let $i$ be the integer with $0\le i \le \ell-1$ and $\varGamma_{H}$ be the congruence subgroup as  given in Theorem \ref{thm:main}. Then we can take $i$  to  be $0$.
     It follows that   a normalized eigenform $f_2\in S_2(\varGamma_{0}(N\ell))$ with $\rho_{f,\ell}\cong\rho_{f_2,\ell}$ exists if and only if $\varGamma_{H}=\varGamma_{0}(N\ell)$. Then this corollary follows from Theorem \ref{thm:gamma0iffgamma0}.
\end{proof}

\bibliographystyle{amsplain}

\end{document}